\documentclass[a4paper, 11pt]{article}
\usepackage{amssymb}
\usepackage{amsmath}
\usepackage{mathrsfs}
\usepackage{amsfonts}
\usepackage{color}
\usepackage{vmargin}
\usepackage{amsthm}

\def\XXint#1#2#3{{\setbox0=\hbox{$#1{#2#3}{\int}$ }
\vcenter{\hbox{$#2#3$ }}\kern-.6\wd0}}

\long\def\symbolfootnote[#1]#2{\begingroup%
\def\thefootnote{\fnsymbol{footnote}}\footnote[#1]{#2}\endgroup}

\setmarginsrb{20mm}{20mm}{20mm}{20mm}{10mm}{10mm}{10mm}{10mm}

\numberwithin{komcounter}{section}

\usepackage{array}
\usepackage[english]{babel}
\usepackage{amssymb,amsthm}
\usepackage[all]{xy}
\newtheorem{defi}{Definition}
\newtheorem{lem}{Lemma}
\newtheorem{prop}{Proposition}

\newtheorem{tw}{Theorem}
\input epsf
\theoremstyle{definition}

\newtheorem{rem}{Remark}

\long\def\symbolfootnote[#1]#2{\begingroup%
\def\thefootnote{\fnsymbol{footnote}}\footnote[#1]{#2}\endgroup}

\setmarginsrb{20mm}{20mm}{20mm}{20mm}{10mm}{10mm}{10mm}{10mm}

\newtheoremstyle{remark}
  {}{}{}{}{\bfseries}{.}{.5em}{{\thmname{#1 }}{\thmnumber{#2}}{\thmnote{ (#3)}}}
\theoremstyle{remboldstyle}

\usepackage{amsmath,amssymb}
\usepackage{amssymb,amsthm}

\title{ \bf The time fractional Schr\"{o}dinger equation on Hilbert space}

\author{ Przemys{\l}aw
 G\'orka$^1$,  Humberto Prado$^2$ \& Juan Trujillo$^3$ \\
\small{$^{1}$ Department of Mathematics and Information Sciences,}\\
\small{Warsaw University of Technology,}\\
\small{Ul. Koszykowa 75, 00-662 Warsaw, Poland.}\\
{\tt pgorka@mini.pw.edu.pl}\\
\small{$^{2}$ Departamento de Matem\'atica y Ciencia de la
Computaci\'on,}\\
\small{ Universidad de Santiago de Chile }\\
\small{Casilla 307 Correo 2, Santiago, Chile.}\\
\small{$^{3}$ Departamento de An\'{a}lisis Matem\'{a}tico,}\\
\small{ Universidad de la Laguna }\\
\small{La Laguna, Tenerife, Spain.}}

\begin{document}
\maketitle
\begin{abstract}
We study the linear fractional Schr\"{o}dinger equation on a Hilbert space, with a fractional time derivative of order $0<\alpha<1,$ and a self-adjoint generator $A.$
Using the spectral theorem we prove existence and uniqueness of strong solutions, and we show that the solutions are governed
by an operator solution family $\{ U_{\alpha}(t)\}_{t\geq 0}$. Moreover, we prove that the solution family  $U_{\alpha}(t)$ converges strongly
to the family of unitary operators $e^{-itA},$  as $\alpha$ approaches to $1$.
\end{abstract}
\bigskip

\medskip

\section{Introduction}
The Schr\"{o}dinger equation is the basic equation of quantum mechanics, it describes the evolution
in time of a quantum system. More recently, N. Laskin has introduced the fractional  Schr\"{o}dinger equation, as a result of extending the Feynman path integral, the resulting equation is a fundamental equation in
fractional quantum mechanics \cite{Laskin1, Laskin2, Laskin3}. Furthermore, N. Laskin \cite{Laskin1} states that ``the fractional Schro¨dinger equation provides us with a
general point of view on the relationship between the statistical
properties of the quantum mechanical path and the
structure of the fundamental equations of quantum mechanics".
Naber \cite{Naber} introduced and examined some
properties of the time-fractional Schr\"{o}dinger equation,
\begin{eqnarray}\label{nonl}
   \frac{\partial^{\alpha} u}{\partial t^{\alpha}}(t) &=& (-i)^{\alpha}A u (t)  \\
     u(0) &=& u_0, \quad\quad \quad \quad \quad \quad\quad\nonumber
\end{eqnarray}
in which $(-i)^{\alpha} = e^{-i \alpha \frac{\pi}{2}}$.
It was shown in \cite{Naber} that the above equation (\ref{nonl}) is equivalent to the usual Schr\"{o}dinger equations with a time dependent Hamiltonian. On the other hand, it was point out that the so-called quantum comb model \cite{Baskin, Bayin, Iomin1, Iomin2},
leads to a time-fractional Schr\"{o}dinger equation with $\alpha=\frac{1}{2}$. Equation (\ref{nonl}) describes non-Markovian evolution in Quantum Mechanics. As a result this
system has memory.
Different aspects of the time fractional Schr\"{o}dinger equation have already been studied.
Particular solutions were sought in \cite{Bayin, Dong, Naber} and numerical analysis performed in \cite{Ford}.
Nevertheless, to the best of our knowledge  there are no results in the literature which show in full generality the uniqueness and existence of solutions to the abstract
Schr\"{o}dinger equation on a Hilbert space.

The purpose of this paper is to consider the abstract fractional evolution equation (\ref{nonl}) on a Hilbert space $\cal{H},$
in which $A$ is a positive self adjoint operator  on $\cal{H},$ and  $\frac{\partial^{\alpha} u}{\partial t^{\alpha}}$ is
the Caputo fractional derivative of order $\alpha\in (0,1).$ We show that $A$ generates a family of bounded operators
$\{U_{\alpha}(t) )\}_{t\geq 0}$ which are defined by the functional calculus of  $A$ via the Mittag-Leffler function
when evaluated at $A.$ Moreover if $u_0$ belongs to the domain of $A$ then we show that $u(t)=U_{\alpha}(t)u_0$ is the
unique strong solution of problem (\ref{nonl}). We also study the problem of the continuous dependence on $\alpha$ for
$U_{\alpha}(t),$ and we show that
\[
  \displaystyle\lim_{\alpha\to 1^{-}} U_{\alpha}(t)=e^{-itA},
\]
where $e^{-itA}$ is the
unitary group whose infinitesimal generator corresponds to the self adjoint operator $A$ and in this case we have, as a
limiting process, the classical Theorem of Stone.

The remainder of the paper is structured as follows. In Section 2, we introduce the notations and
recall the notion of the Caputo derivative. We also give the definition of strong solution to the fractional Schr\"{o}dinger equation. Moreover,
we formulate and prove some technical but very crucial lemma. The main result about existence and uniqueness of solution in shown in Section 3.
The properties of the solution operator are formulated and proven in Section 4.

\section{Preliminaries}
We use the standard notation
$$g_{\alpha}(t)=\displaystyle\frac{t^{\alpha-1}}{\Gamma(\alpha)},\quad \mbox{for}\quad \alpha>0,\quad  t>0.$$
We recall the
definition of the Riemmann Liouville integral by  the convolution product,
 $$J^{\alpha}f(t)=\frac{1}{\Gamma(\alpha)}\displaystyle\int_0^t(t-s)^{\alpha-1}f(s)ds,$$ for a given locally integrable
function  $f$  defined on  the half line $\mathbb{R}_+=[0,\infty)$ and taking values on a Banach space $X.$ Henceforth we use the notation,
\begin{eqnarray*}
J^{\alpha}f(t)=(g_{\alpha}*f)(t).
\end{eqnarray*}

\noindent Then the following property holds
\begin{equation}\label{RL1} J^{\alpha+\beta}f=J^{\alpha}J^{\beta}f,\quad\quad\quad \mbox{for}\quad\quad \alpha,\,\, \beta>0,
\end{equation}
in which $f$ is suitable enough.

 Hereafter we will consider  the following definition of the fractional derivative of order $\alpha\in (0,1).$  Assume  that $u\in C([0,\infty);X)$ and that the convolution
$g_{1-\alpha}*u$\,\, belongs to\,\, $C^1((0,\infty);X).$ Then the Caputo fractional derivative of order $\alpha\in (0,1),$ can be interpreted as
$$
  D^{\alpha} u(t)=\frac{d}{dt}(g_{1-\alpha}*u)(t)-u(0)g_{1-\alpha}(t)=\frac{1}{\Gamma(1-\alpha)}\left[\frac{d}{dt}\left(\int_0^t(t-s)^{-\alpha}u(s)ds\right)-\frac{u(0)}{t^{\alpha}}\right].
$$
 Furthermore if $u\in AC([0,\infty);X),$ in which $AC([0,\infty);X)$ is the space of absolutely continuous functions on $[0,\infty),$ then we can also realize the Caputo derivative as
\begin{equation}\label{cap1} D^{\alpha}u(t)=J^{1-\alpha}u'(t)\quad\quad\mbox{for}\quad\quad 0<\alpha<1;
\end{equation}
see \cite{Baz00,KST} for further properties and definitions.\\
 Henceforth we shall denote the Caputo derivative either by $D^{\alpha} u(t)$\,\, or\,\, $\displaystyle\frac{\partial^{\alpha} u}{\partial t^{\alpha}}(t),$ indistinctly.

\rem\label{rmk1}\rm{ We let $E_{\alpha}(z)$\, be the Mittag-Leffler function, that is,
$$E_{\alpha}(z)=\sum_{k=0}^{\infty} \frac{z^k}{\Gamma (\alpha k+1)}\,\,\,\,\,\,\,(\alpha>0,\,\,\,z\in\mathbb C).$$   Let $X$ be a Banach space, and suppose that $u_0\in X$ and $\omega\in \mathbb C.$ If $0 < \alpha < 1.$
 Then the equation
 \begin{equation}\label{ex}D_t^{\alpha}u(t) =\omega u(t),\,\,\,\, u(0) =u_0,
\end{equation}
 has a unique solution given by  $$u(t)=u_0E_{\alpha}(\omega t^{\alpha})$$
see \cite{Baz00,GIL, KST}. Moreover, the uniqueness of the solution of (\ref{ex}) follows by the uniqueness theorem for the Laplace transform.

Let $A$ be a  densely defined self-adjoint operator on a Hilbert space $\cal{H},$ and  let $0<\alpha<1.$  For a given $u_0\in \cal{H}$ we study the following equation of fractional order $\alpha$

\begin{eqnarray}\label{ev1}
    \frac{\partial^{\alpha} u}{\partial t^{\alpha}}(t) &=& (-i)^{\alpha}A u (t), \quad\quad\quad t>0,\\
     u(0) &=& u_0. \quad\quad \quad \quad \quad \quad\,\,\,\nonumber
\end{eqnarray}
We first introduce the notion of strong solution for the abstract fractional Cauchy problem (\ref{ev1}).
\begin{defi} Let $0<\alpha<1.$ Assume that $u_0\in D(A).$ A function $u$ is called a strong solution of (\ref{ev1}) if $u\in C(\mathbb R_+;D(A))$ and $g_{1-\alpha}*u$\,\, belongs to\,\, $C^1((0,\infty);\cal{H}),$
 and  (\ref{ev1}) holds for all $t>0.$
\end{defi}

We will show that the strong solution of (\ref{ev1}) is determined by the functional  calculus for a self-adjoint operator when its applied to the Mittag-Leffler function. Moreover, the following lemma will give us the necessary bounds we need in the proof of the qualitative properties of the solution operator.\\

In order to prove the next lemma we recall from  \cite[Theorem 2.3 equation 26] {GLL}} that the Mittag-Leffler function has the following representation for $\alpha\in (0,1],$
\begin{equation}\label{mitag1}
	E_{\alpha}(z)=\int_0^{\infty} K_{\alpha}(r,z)dr+\frac{1 }{\alpha}e^{z^{1/\alpha}},\quad\quad |arg(z)|<\pi/\alpha\quad \mbox{and}\quad z\neq 0,
\end{equation}
in which
\begin{equation}\label{K0} K_{\alpha}(r,z)=-\frac{e^{-r^{1/\alpha}}z\, \sin(\pi\alpha)}
{\pi\alpha(r^2-2rz \cos(\pi\alpha)+z^2)}.
\end{equation}

\begin{lem}\label{lema3}
 $a)$ Let $\alpha_0\leq \alpha < 1/2$, in which $\alpha_0>0.$ Then there is a positive constant $ M(\alpha)$ such that for all $t\geq 0$
\begin{equation}\label{b1}
	\displaystyle\sup_{\omega \geq 0} |E_{\alpha}((-it)^{\alpha}\, \omega)|\leq M(\alpha).	
\end{equation}
$b)$ There is $M>0$ such that for all $t\geq 0$ and all $\alpha\in [1/2,1)$

\begin{equation}\label{b2}
\displaystyle\sup_{\omega \geq 0}|E_{\alpha}((-it)^{\alpha}\, \omega)|\leq M.
\end{equation}
 \end{lem}

\begin{proof}
First we show (\ref{b1}) that  is $\alpha_0\leq \alpha<1/2$. We notice that it suffices to prove assertion  (\ref{b1}) for $t=1$. Indeed, let us assume that (\ref{b1}) holds for $t=1$, then for any $t>0$ we have that,
\[
	|E_{\alpha}((-it)^{\alpha}\, \omega)| = |E_{\alpha}((-i)^{\alpha}\, t^{\alpha}\omega)| \leq M(\alpha).
\] To begin we assume  $\omega\geq 1/\alpha_0.$
 Next we recall that $(-i)^{\alpha}=e^{-i\alpha\pi/2}.$ Then we proceed to estimate $|K_{\alpha}(r,(-i)^{\alpha}\omega)|$ for arbitrary  $\omega\geq 1/ \alpha_0.$ Thus,
\begin{equation}\label{ineq1}|K_{\alpha}\left(r,(-i)^{\alpha}\omega\right)|\leq \frac{B}{\pi\alpha}\frac{e^{-r^{1/\alpha}}\omega}
{|(r^2-2r(a-ib)\omega A+\omega^2(A-iB)|},
\end{equation}
where $A=\cos(\pi\alpha),$\, $B=\sin(\pi\alpha),$\, $a=\cos(\alpha\pi/2),$\, $b=\sin(\alpha\pi/2)$ and these quantities are all positive since $0<\alpha<1/2.$ Next we set $u(r)$ and $v(r)$ the real  and imaginary parts respectively  of the denominator on the right hand side of (\ref{ineq1}), that is $$u(r)=r^2-2ra\omega A+\omega^2 A,\,\,\mbox{ and}\,\,\, v(r)=2rb\omega A-\omega^2 B.$$
Hence,
 \begin{equation}\label{ineq2}
 |K_{\alpha}\left(r,(-i)^{\alpha}\omega\right)|\leq \frac{B}{\pi\alpha}\frac{e^{-r^{1/\alpha}}\omega}
{|(r^2-2r(a-ib)\omega A+\omega^2(A-iB)|}\leq \frac{Be^{-r^{1/\alpha}}\omega}{\pi\,\alpha\, |\,u(r)\,|}
\end{equation}
 On the other hand the quadratic $u(r)=r^2-2ra\omega A+\omega^2 A$ is positive for all real $r$ and its minimum  equals to $w^2 A(1-a^2A)>0$ since $a^2A<1,$ and $A>0.$ But then, $$u(r)=r^2-2ra\omega A+\omega^2 A\geq w^2 A(1-a^2A)>0.$$ Hence the right side of (\ref{ineq2}) turns out to be less than or equals to
\begin{equation}\label{ineq3}
 \frac{Be^{-r^{1/\alpha}}\omega}{\pi\,\alpha\, w^2 A(1-a^2A)}
 \end{equation}
Therefore, from (\ref{ineq2}) and (\ref{ineq3})  follows that
\begin{equation*}\label{ineq4}
|K_{\alpha}\left(r,(-i)^{\alpha}\omega\right)|\leq \frac{e^{-r^{1/\alpha}}}{\pi A(1-a^2A)},
\end{equation*}
 since $0<\alpha_0\leq \alpha<1/2$ and $\omega\geq 1/\alpha_0.$ Furthermore

 \begin{eqnarray}\label{ineq5}
|K_{\alpha}(r,(-i)^{\alpha}\omega)|\leq \left\{ \begin{array}{lc}
\displaystyle\frac{e^{-r}}{\pi A(1-a^2A)}& \textrm{ for } r >1\\\

\displaystyle\frac{1}{\pi A(1-a^2A)} & \textrm{ for } r \leq 1.
\end{array} \right.
\end{eqnarray}
Therefore, from  (\ref{ineq5}) we obtain that the integral
\begin{equation*}\label{K1}
	\int_0^{\infty} K_{\alpha}(r,(-i)^{\alpha}\omega)dr,
\end{equation*}
is bounded independently of  $\omega\geq \frac{1}{\alpha_0}$. But then, it follows from the integral representation (\ref{mitag1}) that there is a bound $M_1(\alpha)$ such that,
\begin{equation}\label{b3}
	\displaystyle\sup_{\omega \geq 1/\alpha_0} |E_{\alpha}((-i)^{\alpha}\, \omega)|\leq M_1(\alpha).	
\end{equation}

Now if $\omega \leq 1/\alpha_0,$ then we have that
\[
 |(-i)^{\alpha}\omega|\leq \frac{1}{\alpha_0}.
\]
Thus, from the very definition of the Mittag-Leffler function, we obtain that,
\[
	|E_{\alpha}((-i)^{\alpha}\omega)| \leq \sum_{k=0}^{\infty} \frac{\left(\frac{1}{\alpha_0}\right)^k}{\Gamma (\alpha k+1)} =E_{\alpha}(1/\alpha_0).
\]
Moreover, by the Stirling formula
\[
  \Gamma (x) =\sqrt{2 \pi} x^{x-\frac{1}{2}} e^{-x +\frac{\theta}{12x}},
\]
where $\theta \in [0,1]$, we have
\[
  \frac{\left(\frac{1}{\alpha_0}\right)^k}{\Gamma (\alpha k+1)}  \leq \frac{e^2}{\sqrt{2 \pi}} \left(\frac{e}{\alpha_0 ^{\alpha_0 +1} k^{\alpha_0}} \right)^k.
\]
Hence, by the Lebesgue theorem, we obtain that the map $ [\alpha_0, 1] \ni \alpha \mapsto E_{\alpha}(1/\alpha_0)$ is continuous.
Therefore, there exists $M(\alpha_0)$ such that
\begin{equation}\label{E8}
		\sup_{\omega \leq 1/\alpha_0} |E_{\alpha}((-i)^{\alpha}\omega)|\leq   \sup_{\alpha\in [\alpha_0,1/2]} E_{\alpha}(1/\alpha_0) =M (\alpha_0).
\end{equation}
Now, the proof of assertion (\ref{b1})  follows from (\ref{b3}) together with (\ref{E8}).

Next we  show (\ref{b2}). First we assume that $\omega\geq 2$  under the condition $1/2\leq \alpha <1$ from the hypothesis.
Again it suffices to prove assertion  (\ref{b2}) for $t=1$.  We notice that  $A\leq 0,$ and $B,$ $a,$ and $b$ are all positive. Thus $|v(r)|=|2rb\omega A-B\omega^2|=-2rb\omega A+B\omega^2\geq B\omega^2>0.$ Hence,
\begin{equation*}
 |K_{\alpha}\left(r,(-i)^{\alpha}\omega\right)|\leq   \frac{e^{-r^{1/\alpha} }B\,\omega}{\pi\alpha\, |\,v(r)\,|} \leq   \frac{e^{-r^{1/\alpha} }}{\pi\alpha \omega}
 \leq   \frac{e^{-r^{1/\alpha}}}{\pi}.
\end{equation*}
Furthermore,
\begin{eqnarray*}\label{ineq7}
|K_{\alpha}(r,(-i)^{\alpha}\omega)|\leq \left\{ \begin{array}{lc}
\displaystyle\frac{e^{-r}}{\pi}& \textrm{ for } r >1\\\\
\displaystyle\frac{1}{\pi} & \textrm{ for } r \leq 1.
\end{array} \right.
\end{eqnarray*}
Hence, reasoning as in the proof of (\ref{b1}) we obtain that  there is a positive constant $M$ which in this case does not depends on the value of $\alpha\in [1/2,1),$ so that
\begin{equation*}\label{ineq8}\displaystyle\sup_{\omega \geq 2} |E_{\alpha}((-i)^{\alpha}\, \omega)|\leq M.
\end{equation*}
 Now  by an applications of the same argument as in (\ref{E8}) we can show  that there is a $M_2>0$ independent of $\alpha\in [1/2,1)$ such that
 \begin{equation*}\label{ineq9}
 \displaystyle\sup_{\omega \leq 2} |E_{\alpha}((-i)^{\alpha}\, \omega)|\leq M_2.
 \end{equation*}
 Thus the proof of (\ref{b2}) now follows from these last two inequalities.	

\end{proof}

\section{Existence of the dynamics.}
In this part of our paper we state and prove our principal assertion.
\begin{tw}\label{thm1} Let $\mathcal{H}$ be a Hilbert space and let $A$ be a positive self-adjoint operator on $\mathcal{H}.$
Then there exists a unique strong solution to the problem
\begin{eqnarray}\label{stone}
    \frac{\partial^{\alpha} u}{\partial t^{\alpha}}(t) &=& (-i)^{\alpha}A u (t), \quad\quad\quad t>0\\
     u(0) &=& u_0 \quad\quad \quad \quad \quad \quad\,\,\,u_0\in D(A) \quad .\nonumber
\end{eqnarray}
Moreover, there  is a measure space $(\Omega, \mu)$, a Borel
measurable function $a$ on $\Omega$ and a unitary map $W: L^2(\Omega) \rightarrow \mathcal{H}$ such that the unique solution of problem (\ref{stone}) has the following representation
\begin{eqnarray*}
	u(t)=W(E_{\alpha}((-it)^{\alpha}\,a(\cdot\,))W^{-1}u_0.
\end{eqnarray*}
\end{tw}
\begin{proof}

Let us recall that because of the spectral theorem for a self-adjoint operator  $A: D(A)\subset H \rightarrow H$, there exists a measure space
$(\Omega,\mu)$, and a Borel measurable function $a$ and a unitary map $W: L^2(\Omega,\mu)\rightarrow H$ such that the
following diagram commutes
\begin{displaymath}
  \xymatrix{
    L^2 (\Omega, \mu) \ar[r]^{M_a} \ar[d]_W & L^2(\Omega, \mu)  \\
    H \ar[r]_A & H \ar[u]^{W^{-1}}}
\end{displaymath}
for each $f\in L^{2}(\Omega, \mu)$ such that $Wf \in D(A)$. Moreover, if $f\in L^2(\Omega, \mu)$ is given, then $Wf \in D(A)$ if and only if $M_a f \in L^{2}(\Omega, \mu)$; see e.g. \cite{ RS,T}, where $M_a f (x) =a(x)f(x)$.

Thus, the spectral theorem ensure us that there exists a unitary map $W$ from $L^2(\Omega)$ onto $\mathcal{H}$  such that
\begin{equation}\label{mult}
    W^{-1}AW\varphi(\xi)=a(\xi)\varphi(\xi),\quad \xi\in \Omega.
\end{equation}
Now the proof of the theorem falls naturally into two parts.

{\bf Uniqueness.}

Let us assume that $u$ is a strong solution to problem (\ref{stone}).  We define  $v(t,\xi)= (W^{-1}u(t))(\xi)$. Then it follows from (\ref{mult}) that
\begin{equation}
(W^{-1}Au(t))(\xi)=W^{-1}AW v(t)(\xi)=a(\xi) v(t,\xi).
\end{equation}
Let us observe that $g_{1-\alpha}*v \in C^1((0,\infty);L^2(\Omega))$. Indeed, we shall show that
\[
  \frac{d}{dt} g_{1-\alpha}*v = W^{-1}\left(\frac{d}{dt} g_{1-\alpha}*u\right).
\]
Next we set $\Theta (t) = g_{1-\alpha}*v.$ Then using the fact that $W$ is an isometry, we get that
\begin{eqnarray*}
    \left\|\frac{\Theta(t+h)-\Theta(t)}{h}- W^{-1}(\frac{d}{dt} g_{1-\alpha}*u) \right\|_{L^2(\Omega)} \\
=\left\|W^{-1}\left(\frac{g_{1-\alpha}*u (t+h)-g_{1-\alpha}*u(t)}{h}- \frac{d}{dt} g_{1-\alpha}*u(t)\right)\right\|_{L^2(\Omega)}\\
 =  \left\|\frac{g_{1-\alpha}*u (t+h)-g_{1-\alpha}*u(t)}{h}- \frac{d}{dt} g_{1-\alpha}*u(t)\right\|_{\mathcal{H}} \underset {h \rightarrow 0}\longrightarrow 0,
\end{eqnarray*}

\noindent where the convergence follows from the assumptions on the function $u$. Moreover, we can easily check that derivative $\Theta'$ is a continuous function, thus we obtain that $g_{1-\alpha}*v \in C^1((0,\infty);L^2(\Omega))$.
Furthermore, by continuity of $W$ have
\[
  W^{-1}\left(\frac{d}{dt} g_{1-\alpha}*u\right)=\frac{d}{dt} g_{1-\alpha}*W^{-1}u.
\]
Thus, from the  definition of the Caputo derivative we obtain that
 \begin{eqnarray*}
	  W^{-1}\frac{\partial^{\alpha} u}{\partial t^{\alpha}}(t)&=&  W^{-1}\left(\frac{d}{dt} g_{1-\alpha}*u -u_0 g_{1-\alpha} \right)\\
&=&\frac{d}{dt} g_{1-\alpha}* W^{-1}u -v_0 g_{1-\alpha}\\
&=&  \frac{\partial^{\alpha} v(t,\cdot\,)}{\partial t^{\alpha}}.
 \end{eqnarray*}
 Now if we apply $W^{-1}$ to both sides of equation (\ref{stone}), then we obtain the following equation on $L^2(\Omega),$
\begin{eqnarray}\label{lin2}
    \frac{\partial^{\alpha} v(t,\cdot\,)}{\partial t^{\alpha}} &=& (-i)^{\alpha}a(\cdot\,) v(t,\cdot\,), \quad\quad\quad t>0\\
     v(0,\cdot\,) &=&v_0, \nonumber
\end{eqnarray}
where $v_0= W^{-1}u_0$ and $u_0\in D(A)$. Now, it follows from the Remark \ref{rmk1} that the unique solution of the above fractional differential
 equation (\ref{lin2}) is given by
 $$v(t,\xi)=E_{\alpha}((-it)^{\alpha})\,a(\xi))v_0.$$
Since $(W^{-1}u(t))(\xi)=E_{\alpha}((-it)^{\alpha}\,a(\xi))v_0,$ we get that $u$ is given by
\begin{eqnarray}\label{rozwiazanie}
      u(t)=W(E_{\alpha}((-it)^{\alpha}\,a(\cdot\,))W^{-1}u_0), \quad\quad \quad \quad u_0\in D(A).
\end{eqnarray}
This finishes with the proof of the uniqueness property.

{\bf Existence.}

Next, we shall show that $u(t)$ given by formula (\ref{rozwiazanie}) is indeed a strong solution to the initial value problem (\ref{stone}).
First of all, we prove that $u \in C(\mathbb{R}_{+};D(A))$. We need to show that $u(t) \in D(A)$, for all $t\geq 0$. For this purpose let us recall that
  $$Ah=W(a(\cdot\,)(W^{-1}h)(\cdot\,))\,\,\,\,\,\,\mbox{for}\,\,\,\,h\in D(A).$$
Thus, by the spectral theorem we know that $h\in D(A)$ if and only if $a(\cdot\,)\,(W^{-1}h)(\cdot\,)\in L^2(\Omega);$ see \cite{RS,T}. Hence,
$u_0\in D(A)$ if and only if $a(\xi)(W^{-1}u_0)(\xi)$ belongs to $L^2(\Omega).$ But then, from the fact that $\xi\mapsto E_{\alpha}((-it)^{\alpha}\,a(\xi))$ is
bounded by Lemma \ref{lema3},\, it follows that the function $$a(\xi)(W^{-1}u(t))(\xi)=
  E_{\alpha}((-it)^{\alpha}\,a(\xi))a(\xi)(W^{-1}u_0)(\xi),\,\,\,\,\,\xi\in \Omega,$$
  is in $L^2(\Omega)$ for all $t\geq 0$ and effectively we get that $u(t) \in D(A)$.
Moreover, since the mapping $t\mapsto E_{\alpha}((-it)^{\alpha}a(\xi)) $ is continuous, the map $u$ is continuous. Indeed, let us take $t_0, t \in \mathbb{R}_+$,
then we have that

\begin{eqnarray*}
 \left\|W\Big((E_{\alpha}((-i(t+t_0))^{\alpha}\,a(\cdot\,))-E_{\alpha}((-it_0)^{\alpha}\,a(\cdot\,)))W^{-1}\Big)u_0\right\|_{\mathcal{H}}=\\
\left\|\Big(E_{\alpha}((-i(t+t_0))^{\alpha}\,a(\cdot\,))-E_{\alpha}((-it_0)^{\alpha}\,a(\cdot\,))\Big)W^{-1}u_0\right\|_{L^2(\Omega)}.
\end{eqnarray*}
Since $(E_{\alpha}((-i(t+t_0))^{\alpha}\,a(\xi))-E_{\alpha}((-it_0)^{\alpha}\,a(\xi))$ is bounded by Lemma \ref{lema3}. Hence,  there exists $M_{\alpha}$ such that
$$\left|(E_{\alpha}((-i(t+t_0))^{\alpha}\,a(\xi))-E_{\alpha}((-it_0)^{\alpha}\,a(\xi))(W^{-1}u_0)(\xi)\right|\leq M_{\alpha} |(W^{-1}u_0)(\xi)|.$$
Thus, by an application of the Lebesgue dominated convergence theorem the proof of the continuity of the function $u$ defined in (\ref{rozwiazanie}) is finished.

 Next, we prove that the map
\[\Phi(t)=\displaystyle\frac{1}{\Gamma(1-\alpha)}\int_0^t(t-s)^{-\alpha}u(s)ds\]
 belongs to  $C^1((0,\infty); \mathcal{H})$. For this purpose we consider the following mapping
\[ \phi(t)=\frac{1}{\Gamma(1-\alpha)}\int_0^t(t-s)^{-\alpha}E_{\alpha}((-is)^{\alpha}\,a(\xi))ds.\]
Once more by  the definition of Caputo derivative we get
\begin{eqnarray*}
  \phi'(t)=(-i)^{\alpha}a(\xi) E_{\alpha}((-it)^{\alpha}\,a(\xi)) + \frac{1}{\Gamma(1-\alpha)} \frac{1}{t^{\alpha}}.
\end{eqnarray*}
Now, we shall show that
\begin{equation}\label{dif}
  \Phi'(t) = W\phi'(t) W^{-1}u_0.
\end{equation}
Let us notice that
\begin{eqnarray*}
    \lim_{h \rightarrow 0}\left|\left(\frac{\phi(t+h)-\phi(t)}{h}-\phi'(t)\right)W^{-1}u_0\right|=0.
\end{eqnarray*}
Moreover, by the Mean Value Theorem and Lemma \ref{lema3} we have that
\begin{eqnarray*}
    \left|\left(\frac{\phi(t+h)-\phi(t)}{h}-\phi'(t)\right)W^{-1}u_0\right| &=&\left|\left(\frac{1}{h}\int_t^{t+h} \phi'(s) ds-\phi'(t)\right)W^{-1}u_0\right|\\
    &\leq &
C(c_{\alpha}(t) + \left|a(\xi)|)|W^{-1}u_0\right|
\end{eqnarray*}


\noindent for some constants $C$ and $c_{\alpha}(t)$ independent on $h$. Since $u_0 \in D(A)$, we have $a(\xi)W^{-1}u_0, W^{-1}u_0 $ belongs to $L^2(\Omega)$. Moreover,
\begin{equation}\label{dom}\left\|\frac{\Phi(t+h)-\Phi(t)}{h}- W\phi'(t)W^{-1}u_0\right\|_{\mathcal{H}} =
   \left\|W\left(\frac{\phi(t+h)-\phi(t)}{h}- \phi'(t)\right)W^{-1}u_0\right\|_{\mathcal{H}}.
\end{equation}

\noindent Since $W$ is unitary it follows that
\begin{equation}\label{dom1}
\left\|W\left(\frac{\phi(t+h)-\phi(t)}{h}- \phi'(t)\right)W^{-1}u_0\right\|_{\mathcal{H}}=\left\|\left(\frac{\phi(t+h)-\phi(t)}{h}- \phi'(t)\right)W^{-1}u_0\right\|_{L^2(\Omega)}
\end{equation}

\noindent Therefore from (\ref{dom}) and (\ref{dom1}) we obtain  that
\begin{equation*}\left\|\frac{\Phi(t+h)-\Phi(t)}{h}- W\phi'(t)W^{-1}u_0\right\|_{\mathcal{H}} =\left\|\left(\frac{\phi(t+h)-\phi(t)}{h}- \phi'(t)\right)W^{-1}u_0\right\|_{L^2(\Omega)}.
\end{equation*}

\noindent Hence by  Lebesgue dominated convergence we have
$$\left\|\left(\frac{\phi(t+h)-\phi(t)}{h}- \phi'(t)\right)W^{-1}u_0\right\|_{L^2(\Omega)} \underset {h \rightarrow 0}\longrightarrow 0.$$

\noindent Thus, the proof of (\ref{dif}) is complete and hence we have the differentiability of the function $\Phi.$  Furthermore, arguing as above, we get that $\Phi' \in C((0, \infty); \mathcal{H})$.

It remains to prove that the function $u$ defined in (\ref{rozwiazanie}) satisfies equation (\ref{stone}). In order to show this last claim we compute the Caputo derivative of the  $u.$ Thus,
\begin{eqnarray*}
     D^{\alpha} u(t) &=& \Phi'(t) - \frac{1}{\Gamma(1-\alpha)}\frac{u_0}{t^{\alpha}}\\
                     & =& W(-i)^{\alpha}a(\xi) E_{\alpha}((-it)^{\alpha}\,a(\xi))W^{-1}u_0\\
     &=& W(-i)^{\alpha}a(\xi)W^{-1} WE_{\alpha}((-it)^{\alpha}\,a(\xi))W^{-1}u_0\\
     &=& (-i)^{\alpha}Au(t),
\end{eqnarray*}

\noindent and the whole proof of Theorem \ref{thm1} is now finished.
\end{proof}

\begin{rem}\rm{Let $A$ be a self-adjoint operator. Then we shall denote by $U_{\alpha}(t)$ the corresponding solution operator family given by theorem (\ref{thm1}). To be more explicit}
\begin{equation*}\label{st1}
U_{\alpha}(t)\phi= W(E_{\alpha}((-it)^{\alpha}\,a(\cdot\,))W^{-1})\phi,\quad\quad\quad \phi\in \mathcal{H},\,\,t\geq 0.
\end{equation*}
\end{rem}
\section{Properties of the solution operator $U_{\alpha}$}
In this section we study the properties of the solution operator $U_{\alpha}$.
\begin{prop}\label{p1}
The family  $\{U_{\alpha}(t)\}_{t\geq 0}$ satisfy,
\begin{itemize}
\item[(i)]$U_{\alpha}(t)$ is strongly continuous for $t\geq 0$ and $U_{\alpha}(0) = I.$
\item[(ii)]$U_{\alpha}(t)(D(A))\subseteq D(A)$ and $AU_{\alpha}(t)x = U_{\alpha}(t)Ax$ for all $x \in D(A)$\, $t\geq 0.$
\end{itemize}
\end{prop}

\begin{proof}
$(i)$ This follows from the proof of Theorem \ref{thm1}.

$(ii)$ Using similar consideration as in the proof of Theorem \ref{thm1} we get that
\[ U_{\alpha}(t)(D(A))\subseteq D(A).\]
Next, the commutation property $[U_{\alpha}(t),A] = 0$ on $D(A)$  follows from the fact that
$$A=WM_{a(\cdot\,)}W^{-1}, \quad U_{\alpha}(t)=WM_{E_{\alpha}((-it)^{\alpha}\,a(\cdot\,))}W^{-1},\,\,\,\,\,\,\,\,\,t\geq 0.$$
Thus $$AU_{\alpha}(t)\phi= U_{\alpha}(t)A\phi\,\,\,\,\,\,\,\,\mbox{for all}\,\,\phi\in D(A),\,\,t\geq 0.$$
\end{proof}
Next, we state some  further properties of the solution operator $U_{\alpha}$.
\begin{prop}\label{zbi}
Let $\alpha\in (0,1).$ Then the solution operator  enjoys the following properties
\begin{itemize}
\item[(i)] $U_{\alpha}(t)^*=WE_{\alpha}((it)^{\alpha}a(\cdot\,))W^{-1},$\,\,\,\, \,\,\,$t>0.$
\item[(ii)] $U_{\alpha}(t)U_{\alpha}(t)^*=U^{*}_{\alpha}(t)U_{\alpha}(t)=W|E_{\alpha}((it)^{\alpha}a(\cdot\,))|^2W^{-1},$\,\, \,\,\,$t>0.$
\item[(iii)] Let $e^{-itA}$ be the unitary group generated by the self-adjoint operator $A.$ Then $$\displaystyle\lim_{\alpha\to 1^{-}}U_{\alpha}(t)\phi=e^{-itA}\phi, \quad\mbox{ for every}\, \phi\in\mathcal{H},\quad\mbox{and}\quad t\geq 0.$$
\end{itemize}
\end{prop}
\begin{rem}
In the paper of Dong and Xu \cite{Dong} it has been pointed out that the quantity $\left\|U_{\alpha}(t)u_0 \right\|$ is not conserved during the evolution.
\end{rem}

\begin{proof}
(i) Let us take $\phi, \psi \in \mathcal{H}$. Using the fact that $W$ is a unitary operator we get
\begin{eqnarray*}
 (U_{\alpha}(t) \psi, \phi )_{\mathcal{H}}&=& (WE_{\alpha}((-it)^{\alpha}a(\cdot\,))W^{-1} \psi, \phi )_{\mathcal{H}}\\
  &=&(E_{\alpha}((-it)^{\alpha}a(\cdot\,))W^{-1} \psi, W\phi )_{L^2(\Omega)}\\
  &=&\int_{\Omega} E_{\alpha}((-it)^{\alpha}a(x)) W^{-1}\psi(x) \overline{W^{-1}\phi(x)} dx\\
  &=&\int_{\Omega}  W^{-1}\psi(x) \overline{E_{\alpha}((it)^{\alpha}a(x))W^{-1}\phi(x)} dx\\
  &=& (W^{-1}\psi, E_{\alpha}((it)^{\alpha}a(\cdot))W^{-1}\phi)_{L^2(\Omega)}\\
  &=&(\psi, WE_{\alpha}((it)^{\alpha}a(\cdot))W^{-1}\phi)_{\mathcal{H}}.
\end{eqnarray*}
\
Hence, we obtain that,
\[
  U_{\alpha}(t)^*=WE_{\alpha}((it)^{\alpha}a(\cdot\,))W^{-1}.
\]
The proof of $(ii)$ follows from the very definition of $U_{\alpha}(t)$. Next, we show $(iii).$
 We will prove that
 \begin{equation}\label{conv}\displaystyle\lim_{\alpha\to 1^{-}}\|U_{\alpha}(t)\phi- e^{-itA}\phi\|_{\mathcal {H}}=0,\quad\quad t\geq 0,\quad \phi\in\mathcal{H}.
 \end{equation}
Since, $W$ is an isometry and $$e^{-itA}=We^{-it a(\xi)}W^{-1},$$
  we have that
\begin{equation}\label{conv1}
  \|U_{\alpha}(t)\phi- e^{-itA}\phi\|^2_{\mathcal {H}}=\int_{\Omega}\left|E_{\alpha}((-it)^{\alpha}a(\xi))-e^{-it a(\xi)}\right|^2\left|(W^{-1}\phi)(\xi)\right|^2d\xi.
\end{equation}
  According to Lemma \ref{lema3} part $(b),$ for each $t>0$ the function $|E_{\alpha}((-it)^{\alpha}a(\xi))|$ \,is bounded independently of
  $\xi\in M$  and $\alpha\in [1/2,1).$ But then, there is $M$ such that for all $\alpha\in [1/2,1),$ and $\xi\in \Omega$
  $$|E_{\alpha}((-it)^{\alpha}a(\xi))-e^{-it a(\xi)}|\leq M.$$
Hence
\[
  |E_{\alpha}((-it)^{\alpha}a(\xi))-e^{-it a(\xi)}||(W^{-1}\phi)(\xi)|\leq M |(W^{-1}\phi)(\xi)|.
\]
 Moreover
 $$\lim_{\alpha\to 1^{-}}\left(E_{\alpha}((-it)^{\alpha}a(\xi))-e^{-it a(\xi)}\right)=0$$ and
$W^{-1}\phi\in L^{2}(\Omega)$. Then the dominated convergence theorem applies to (\ref{conv1}) when $\alpha\to 1^{-}$, and thus  the proof of (iii) is finished.

\end{proof}

\subsection{An example}

We consider $A\,=-\Delta ,$\,the Laplacian operator on $L^2(\mathbb R^{n}).$ Then by the Spectral Theorem we have that
$$Au:=\mathcal{F}^{-1}(|\xi|^2\mathcal{F})u,\,\,\,\,\,\,\,\,\mbox{for}\,\,\,\,\,\,u\in D(A):=\mathcal{S}(\mathbb R^n).$$ Next we find the strong solution of the following  fractional evolution equation. Suppose that $0<\alpha<1$ and consider the initial value problem
\begin{eqnarray}\label{sch1}
\frac{\partial^{\alpha} u}{\partial t^{\alpha}}(t,x)&=& (-i)^{\alpha}(-\Delta) u(t,x),\,\,\,\,\,t>0,\,\,x\in\mathbb R^n\\
                                                               u(0,\,\cdot\,)&=& g(\,\cdot\,)\in C_0^{\infty}(\mathbb R^n) \nonumber
\end{eqnarray}
We  will show that the strong solution of (\ref{sch1}) is defined by a convolution kernel which is given by the Fourier transform in the distributional sense of the Mittag-Leffler function. To prove this claim we first recall some basic facts. We denote by $\mathcal{S}(\mathbb R^n)$ and by $\mathcal{S}'(\mathbb R^n)$ the Schwartz space and the space of tempered distributions respectively.  Let  $\varphi$ be a function of $\mathcal{S}(\mathbb R^n).$ Then we recall that the action of the dilation operator on  $\varphi$ is defined as $\varphi_{\lambda}(x)= \varphi(\lambda x),$\, $\lambda\in\mathbb R,$ $x\in \mathbb R^n.$ Furthermore the action on the Fourier transform $\mathcal{F}$ is
\begin{equation}\label{d1}(\mathcal{F}\varphi)_{\lambda}=\frac{1}{\lambda^n} \mathcal{F}\varphi_{1/\lambda}\,\,\,\,\,\mbox{and}\,\,\,\,\, \mathcal{F}\varphi_{\lambda}= \frac{1}{\lambda^n}(\mathcal{F}\varphi)_{1/\lambda^n}\,\,\,\,\,\lambda>0.
\end{equation}
If $u$ is a distribution then we recall that\,\, $\langle u_{\lambda},\varphi\rangle=\displaystyle\frac{1}{\lambda^n}\langle u,\varphi_{1/\lambda}\rangle,$ and the same identities as (\ref{d1}) are also verified.\\
Next we set,

$$e(\xi)=E_{\alpha}((-i)^{\alpha}|\xi|^2),\,\,\,\,\,\xi\in\mathbb R^n.\,\,\,$$ Thus,
$$e_{t^{\alpha/2}}(\xi)=E_{\alpha}((-it)^{\alpha}
|\xi|^2).$$
We see that the hypothesis of Theorem \ref{thm1} are satisfied. Hence the strong solution of (\ref{sch1}) is given by
\begin{equation}\label{wzor}
u(t,x)= \mathcal{F}^{-1}(e_{t^{\alpha/2}}(\mathcal{F}g))(x).
\end{equation}
We notice that  the function $\xi\mapsto e_{t^{\alpha/2}}(\xi)$ is bounded for each $t\geq 0$ by Lemma \ref{lema3}. Thus $e_{t^{\alpha/2}}$ defines a tempered distribution by integration. But then,  $\mathcal{F}(e_{t^{\alpha/2}})$  also is a tempered distribution. Now if  $u\in \mathcal{S}'(\mathbb R^n)$ and $\varphi\in \mathcal{S}(\mathbb R^n)$ then   we have that
$u*\varphi\in C^{\infty}(\mathbb R^n)\cap \mathcal{S}'(\mathbb R^n)$ and $(\mathcal{F}u)(\mathcal{F}\varphi)=\mathcal{F}(u*\varphi)$, see \cite{Linares}.
But then we have that

\begin{equation*}
\mathcal{F}^{-1}(e_{t^{\alpha/2}}(\mathcal{F}g))=(\mathcal{F}^{-1}(e_{t^{\alpha/2}})*g)
\end{equation*}
as tempered distributions. Moreover
 \begin{equation*}
 \mathcal{F}^{-1}(e_{t^{\alpha/2}})=\frac{1}{t^{n\alpha/2}} (\mathcal{F}^{-1}e)_{1/t^{\alpha/2}}.
 \end{equation*}
 Hence taking into account the above considerations we can represent the solution of (\ref{sch1}) as
 \begin{eqnarray*}
 u(t,x)= \frac{1}{t^{n\alpha/2}}((\mathcal{F}^{-1}e)_{1/t^{\alpha/2}}*g)(x).
 \end{eqnarray*}
 Since $e_{t^{\alpha/2}} \in C^{\infty} \cap L^{\infty}$, we have that  $e_{t^{\alpha/2}}(\mathcal{F}g) \in \mathcal{S}$.
 Thus, from formula (\ref{wzor}), we get that the function $x\mapsto u(t,x)$ belongs to the Schwartz space for each $t \geq 0$.

 Using Proposition \ref{zbi} and the above considerations we close the paper with the following observation.
\begin{prop}
  Let $\phi \in \mathcal{S}(\mathbb{R}^n)$, then
    \[
      \frac{1}{t^{n\alpha/2}}((\mathcal{F}^{-1}e)_{1/t^{\alpha/2}}*\phi) \overset{L^2}{\underset{\alpha \rightarrow 1}\longrightarrow} \frac{1}{(4 \pi i t)^{\frac{n}{2}}} \int_{\mathbb{R}^n}
      e^{i \frac{|\cdot-y|^2}{4t}} \phi(y) dy.
    \]

\end{prop}

\subsection*{Acknowledgements}
During the preparation of this manuscript H.Prado and  P. G\'orka enjoyed the support of FONDECYT grant \# 1130554, and MECESUP, USA 1298,\#596.
J. Trujillo has been partially supported the FEDER fund and by project MTM2013-41704-P from the Government of Spain. Some part of the paper has
been performed during the visit of P.G. to the USACH in Santiago de Chile. P.G. wish to thank for its hospitality. The
authors would like to thank the referee for his comments and pointing out the mistake in the proof of Lemma 1 in the first version of the manuscript.

\end{document}